\documentclass[12pt,a4paper]{article}
\usepackage{indentfirst}
\usepackage{amsmath,amssymb,amsfonts,amsthm,graphics}
\usepackage{makeidx}
\usepackage{amsmath,amssymb,amsfonts,amsthm,graphics,graphicx}
\usepackage{makeidx}
\usepackage{color}
\usepackage{ifpdf}
\usepackage{subfigure}

\newtheorem{thm}{Theorem}

\newtheorem{prob}{Problem}
\newtheorem{lem}{Lemma}
\newtheorem{pro}{Proposition}

\theoremstyle{definition}

\def\-{\mbox{--}}

\newtheorem{obser}{Observation}

\def\pf{\noindent {\it Proof.} }
\begin{document}

\title{\Large\bf On (strong) proper vertex-connection of graphs\footnote{Supported by NSFC No.11371205 and PCSIRT.} }
\author{\small Hui Jiang, Xueliang Li, Yingying Zhang, Yan Zhao\\
\small Center for Combinatorics and LPMC-TJKLC\\
\small Nankai University, Tianjin 300071, China\\
\small E-mail: jhuink@163.com; lxl@nankai.edu.cn;\\
\small zyydlwyx@163.com; zhaoyan2010@mail.nankai.edu.cn}
\date{}
\maketitle
\begin{abstract}

A path in a vertex-colored graph is a {\it vertex-proper path} if
any two internal adjacent vertices differ in color. A vertex-colored
graph is {\it proper vertex $k$-connected} if any two vertices of
the graph are connected by $k$ disjoint vertex-proper paths of the
graph. For a $k$-connected graph $G$, the {\it proper vertex
$k$-connection number} of $G$, denoted by $pvc_{k}(G)$, is defined
as the smallest number of colors required to make $G$ proper vertex
$k$-connected. A vertex-colored graph is {\it strong proper
vertex-connected}, if for any two vertices $u,v$ of the graph, there
exists a vertex-proper $u$-$v$ geodesic. For a connected graph $G$,
the {\it strong proper vertex-connection number} of $G$, denoted by
$spvc(G)$, is the smallest number of colors required to make $G$
strong proper vertex-connected. These concepts are inspired by the
concepts of rainbow vertex $k$-connection number $rvc_k(G)$, strong
rainbow vertex-connection number $srvc(G)$, and proper
$k$-connection number $pc_k(G)$ of a $k$-connected graph $G$.
Firstly, we determine the value of $pvc(G)$ for general graphs and
$pvc_k(G)$ for some specific graphs. We also compare the values of
$pvc_k(G)$ and $pc_k(G)$. Then, sharp bounds of $spvc(G)$ are given
for a connected graph $G$ of order $n$, that is, $0\leq spvc(G)\leq
n-2$. Moreover, we characterize the graphs of order $n$ such that
$spvc(G)=n-2,n-3$, respectively. Finally, we study the relationship
among the three vertex-coloring parameters, namely, $spvc(G), \
srvc(G)$ and the chromatic number $\chi(G)$ of a connected graph
$G$.

{\flushleft\bf Keywords}: vertex-coloring, proper vertex connection,
strong proper vertex connection.

{\flushleft\bf AMS subject classification 2010}: 05C15, 05C40,
05C38, 05C75.
\end{abstract}

\section{Introduction}

In this paper, all graphs considered are simple, finite and undirected. We refer to
book \cite{B} for undefined notation and terminology in graph theory. For simplicity,
a set of internally vertex-disjoint paths will be called {\it disjoint}. A path in
an edge-colored graph is a {\it rainbow path} if its edges have different colors. An
edge-colored graph is {\it rainbow $k$-connected} if any two vertices of the graph
are connected by $k$ disjoint rainbow paths of the graph. For a $k$-connected graph
$G$, the {\it rainbow $k$-connection number} of $G$, denoted by $rc_{k}(G)$, is defined
as the smallest number of colors required to make $G$ rainbow $k$-connected. These concepts
were first introduced by Chartrand et al. in \cite{CJM, CJMZ}. Since then, a lot of
results on the rainbow connection have been obtained; see \cite{LSS, LSu}.

As a natural counterpart of the concept of rainbow $k$-connection, the concept of
rainbow vertex $k$-connection was first introduced by Krivelevich and Yuster
in \cite{KY} for $k=1$, and then by Liu et al. in \cite{LMS} for general $k$.
A path in a vertex-colored graph is a {\it vertex-rainbow path} if its internal
vertices have different colors. A vertex-colored graph is {\it rainbow vertex $k$-connected}
if any two vertices of the graph are connected by $k$ disjoint vertex-rainbow paths of the
graph. For a $k$-connected graph $G$, the {\it rainbow vertex $k$-connection number} of
$G$, denoted by $rvc_k(G)$, is defined as the smallest number of colors required to make
$G$ rainbow vertex $k$-connected. There are many results on this topic, we refer to \cite{CLS1,CLS2,LS}.

In 2011, Borozan et al. \cite{BFG} introduced the concept of proper
$k$-connection of graphs. A path in an edge-colored graph is a {\it
proper path} if any two adjacent edges differ in color. An
edge-colored graph is {\it proper $k$-connected} if any two vertices
of the graph are connected by $k$ disjoint proper paths of the
graph. For a $k$-connected graph $G$, the {\it proper $k$-connection
number} of $G$, denoted by $pc_{k}(G)$, is defined as the smallest
number of colors required to make $G$ proper $k$-connected. Note
that $$1\leq pc_k(G)\leq \min\{\chi'(G), rc_k(G)\},\ \ \ \  \ \ \ \
\ \ \ \ \ \  \ \  (1)$$ where $\chi'(G)$ denotes the edge-chromatic
number. Recently, the case for $k=1$ has been studied by Andrews et
al. \cite{ALL} and Laforge et al. \cite{LLZ}.

Inspired by the concepts above, now we introduce the concept of
proper vertex $k$-connection. A path in a vertex-colored graph is a
{\it vertex-proper path} if any two internal adjacent vertices
differ in color. A vertex-colored graph is {\it proper vertex
$k$-connected} if any two vertices of the graph are connected by $k$
disjoint vertex-proper paths of the graph. For a $k$-connected graph
$G$, the {\it proper vertex $k$-connection number} of $G$, denoted
by $pvc_{k}(G)$, is defined as the smallest number of colors
required to make $G$ proper vertex $k$-connected. We write $pvc(G)$
for $pvc_{1}(G)$, and similarly, $rc(G), rvc(G)$ and $pc(G)$ for
$rc_1(G), rvc_1(G)$ and $pc_1(G)$. Note that $$0\leq pvc_k(G)\leq
\min\{\chi(G), rvc_k(G)\},\ \ \ \  \ \ \ \ \ \ \ \ \ \  \ \  (2)$$
where $\chi(G)$ denotes the chromatic number of $G$. By Brooks'
theorem \cite{B}, if $G$ is a connected graph and is neither an odd
cycle nor a complete graph, then $\chi(G) \leq \Delta$, and so
$pvc_k(G)\leq\Delta$, where $\Delta$ denotes the maximum degree of
$G$.

\section{Proper vertex $k$-connection}

In this section, we determine the value of $pvc(G)$ for general
graphs and $pvc_k(G)$ when $G$ is a cycle, a wheel, and a complete
multipartite graph. Moreover, we show that $pc_k(G) \geq pvc_k(G)$
for $k=1$ and provide an example graph $G$ such that $pc_k(G)>
pvc_k(G)$ for $k\geq 2$.

\subsection{Proper vertex-connection number $pvc(G)$}

From the definition of $pvc(G)$, the following results are
immediate. Recall that the {\it diameter} of a connected graph $G$,
denoted by $diam(G)$, is the maximum of the distances among pairs of
vertices of $G$.

\begin{pro}\label{pro1} Let $G$ be a nontrivial connected graph. Then

$(a)$  \ $pvc(G)=0$ if and only if $G$ is a complete graph;

$(b)$ \ $pvc(G)=1$ if and only if $diam(G)=2$.
\end{pro}

For the case that $diam(G)\geq 3$, we have the following theorem.

\begin{thm}\label{thm1} Let $G$ be a nontrivial connected graph. Then, $pvc(G)=2$ if
and only if $diam(G)\geq 3$.
\end{thm}

\pf The necessity can be verified by Proposition \ref{pro1}.

Now we prove its sufficiency. Since $diam(G)\geq3$, we have that
$pvc(G)\geq2$ and then we just need to prove that $pvc(G)\leq2$. Let
$T$ be a spanning tree of $G$. For a vertex $v\in V(T)$, let
$e_{T}(v)$ denote the eccentricity of $v$ in $T$, i.e., the maximum
of the distances between $v$ and the other vertices in $T$. Let
$V_{i}=\{u\in V(T) : d_T(u,v)=i\}$, where $0\leq i\leq e_{T}(v)$.
Hence $V_{0}=\{v\}$. Define a $2$-coloring of the vertices of $T$ as
follows: If $i$ is odd, color the vertices of $V_{i}$ with color 1;
otherwise, color the vertices of $V_{i}$ with color 2. It is easy to
check that for any two vertices $x$ and $y$ in $G$, there is a
vertex-proper path connecting them. Thus, $pvc(G)\leq2$, and
therefore, $pvc(G)=2$.\qed

\subsection{Proper vertex $k$-connection of some specific graphs}

In this subsection, we shall determine the value of $pvc_{k}(G)$ for
some specific graphs. Let $\kappa(G)=$ max\{$k:G$ is $k$-connected\}
denote the vertex-connectivity of $G$. Note that $pvc_{k}(G)$ is
well defined if and only if $1\leq k\leq\kappa(G)$. We start with
the case that $G$ is a cycle of order $n$, denoted by $C_n$. Observe
that $\kappa(C_n)=2$. We have the following results.

\begin{thm}\label{thm2} \

$(a)$ \ $pvc(C_{3})=0$, $pvc(C_{4})=pvc(C_{5})=1$, and
$pvc(C_{n})=2$ for $n\geq 6$.

$(b)$ \ $pvc_{2}(C_{3})=1$, $pvc_{2}(C_{n})=2$ for $n\geq 4$ even,
and $pvc_{2}(C_{n})=3$ for $n\geq 5$ odd.
\end{thm}

\pf $(a)$ \ Since $C_3=K_3$ and $diam(C_4)=diam(C_5)=2$, we have
that $pvc(C_{3})=0$ and $pvc(C_{4})=pvc(C_{5})=1$ by Proposition
\ref{pro1}. Since $diam(C_n)\geq3$ for $n\geq6$, it follows that
$pvc(C_{n})=2$ $(n\geq6)$ by Theorem \ref{thm1}.

$(b)$ \ The assertion can be easily verified for $C_3$. Now, let
$n\geq 4$. We consider two cases, depending on the parity of $n$.

Case 1. $n$ is even. By $(2)$, we have that
$pvc_2(C_n)\leq\chi(C_n)=2$. If one colors the vertices of $C_n$
with one color, then we do not have two vertex-proper paths between
any two adjacent vertices. Hence, $pvc_2(C_n)=2$ for $n\geq 4$ even.

Case 2. $n$ is odd. Similarly from $(2)$, it follows that
$pvc_2(C_n)\leq\chi(C_n)=3$. Assume that $C_n=v_1v_2\cdots v_nv_1$
$(n\geq 5)$. If we have a vertex-coloring for $C_n$ with two colors,
then there must exist two adjacent vertices, say $v_1$ and $v_2$,
colored the same. However, there do not have two vertex-proper paths
between $v_n$ and $v_3$. Thus, $pvc_2(C_n)=3$ for $n\geq 5$ odd.
\qed

A graph obtained from $C_{n}$ by joining a new vertex $v$ to every vertex of $C_{n}$ is
the {\it wheel} $W_{n}$. The vertex $v$ is the center of $W_{n}$. Note that $\kappa(W_{n})=3$.

\begin{thm}\label{thm3}\

$(a)$ \ $pvc(W_{3})=0$ and $pvc(W_{n})=1$ for $n\geq 4$.

$(b)$ \ $pvc_{2}(W_{3})=1$ and $pvc_{2}(W_{n})=pvc(C_{n})$ for
$n\geq 4$.

$(c)$ \ $pvc_{3}(W_{3})=1$ and $pvc_{3}(W_{n})=pvc_{2}(C_{n})$ for
$n\geq 4$.
\end{thm}

\begin{proof} $(a)$ \ Since $W_3=K_4$ and $diam(W_n)=2$ for $n\geq 4$, we have that
$pvc(W_{3})=0$ and $pvc(W_{n})=1$ $(n\geq 4)$ by Proposition \ref{pro1}.

$(b)$ \ The assertion can be easily verified for $W_{3}$. Now, let
$n\geq4$. Take a proper vertex connected coloring for the cycle
$C_{n}$ in $W_{n}$ with $pvc(C_{n})$ colors and then color the
center with any used color. Clearly, $W_{n}$ is proper vertex
2-connected. Thus, $pvc_{2}(W_{n})\leq pvc(C_{n})$. On the other
hand, consider a vertex-coloring for $W_{n}$ with fewer than
$pvc(C_{n})$ colors. Then, there exist two vertices $u,v$ in the
cycle $C_{n}$ of $W_{n}$ such that we do not have a vertex-proper
$u$-$v$ path along the cycle. Hence, there is at most one
vertex-proper $u$-$v$ path in $W_{n}$ (using the center of $W_{n}$).
Thus, $pvc_{2}(W_{n})\geq pvc(C_{n})$.

$(c)$ \ This can be proved by a similar way as the proof of Theorem
\ref{thm3}$(b)$.
\end{proof}

For the complete graph $K_n$, we have that $pvc(K_n)=0$ and
$pvc_2(K_n)=pvc_3(K_n)=\cdots = pvc_{n-1}(K_n)=1$. Let $K_{n_1,n_2}$
denote the complete bipartite graph, where $2\leq{n_1}\leq{n_2}$.
Clearly, $\kappa(K_{n_1,n_2})=n_1$. Then, we have
$pvc(K_{n_1,n_2})=1$ and $pvc_k(K_{n_1,n_2})=2$ for $2\leq k\leq
n_1$. Let $G= K_{n_1,\ldots,n_t}$ be a complete multipartite graph,
where $1 \leq n_1 \leq \cdots \leq n_t$ with $t\geq 3$ and $n_t\geq
2$. In \cite{LMS}, Liu, Mestre and Sousa determined the rainbow
vertex $k$-connection number of $K_{n_1,\ldots,n_t}$. By the same
method as the proof of Theorem 4 in \cite{LMS} and the fact that
$pvc_k(G) \leq rvc_k(G)$, we deduce the following result.

\begin{thm}\label{thm5} Let $1\leq n_{1}\leq\cdots \leq n_{t}$,
where $t\geq 3$,$n_{t}\geq 2$ and $m=\sum_{i=1}^{t-1}n_{i}$.

$(a)$ \ If $1\leq k\leq m-2$, then we have the following:

~~~    $(\romannumeral1) \ pvc_{k}(K_{n_{1},...,n_{t}})=1$ if \
$1\leq k\leq m-n_{t-1}+1$.

~~~    $(\romannumeral2) \ pvc_{k}(K_{n_{1},...,n_{t}})=2$ if \
$m-n_{t-1}+2\leq k\leq m-2$.

$(b)$ \ $(\romannumeral1) \ pvc_{m-1}(K_{n_{1},...,n_{t}})=1$ if \
$n_{t-1}\leq 2$.

~~~    $(\romannumeral2) \ pvc_{m-1}(K_{n_{1},...,n_{t}})=2$ if \
$n_{t-1}\geq 3$ and we do not have $n_{t}=n_{t-1}=n_{t-2}$ odd.

~~~    $(\romannumeral3) \ pvc_{m-1}(K_{n_{1},...,n_{t}})=3$ if \
$n_{t}=n_{t-1}=n_{t-2}\geq 3$ are odd.

$(c)$ \ $(\romannumeral1) \ pvc_{m}(K_{n_{1},...,n_{t}})=1$ if \
$n_{t-1}=1$.

~~~    $(\romannumeral2) \ pvc_{m}(K_{n_{1},...,n_{t}})=2$ if \
$2\leq n_{t-1}\leq n_{t}-2$.

~~~    $(\romannumeral3) \ pvc_{m}(K_{n_{1},...,n_{t}})=2$ if \
$n_{t-1}=n_{t}-1\geq2$ and $n_{t-2}\leq 2$, or $n_{t-1}=n_{t}\geq2$

~~~~~~~~~~and $n_{t-2}=1$.

~~~    $(\romannumeral4) \ pvc_{m}(K_{n_{1},...,n_{t}})=3$ if \
$n_{t-1}=n_{t}-1$ and $n_{t-2}\geq3$, or
$n_{t-1}=n_{t}\geq3,n_{t-2}\geq2$

~~~~~~~~~and we do not have $n_{t}=n_{t-1}=n_{t-2}=n_{t-3}=4$ and $t\geq4$.

~~~    $(\romannumeral5) \ pvc_{m}(K_{n_{1},...,n_{t}})=4$ if \
$t\geq4$ and $n_{t}=n_{t-1}=n_{t-2}=n_{t-3}=4$.

 ~~~   $(\romannumeral6) \ pvc_{m}(K_{n_{1},...,n_{t}})=s$ if \
 $n_{t}=n_{t-1}=\cdots=n_{t-s+1}=2$ and $n_{t-s}=n_{t-s-1}=\cdots=$

~~~~~~~~~~$n_{1}=1$, for $1\leq s\leq t$.

\end{thm}

\subsection{Comparing $pc_{k}(G)$ and $pvc_{k}(G)$}

In \cite{KY}, Krivelevich and Yuster compared the values of $rc(G)$
and $rvc(G)$. They observed that one of $rc(G)$ and $rvc(G)$ cannot
be bounded in terms of the other, by providing example graphs $G$
where $rc(G)$ is much larger than $rvc(G)$, and vice versa. In
\cite{LMS}, Liu et al. compared the values of $rc_{k}(G)$ and
$rvc_{k}(G)$, similarly.

Here, we will compare the values of $pc_{k}(G)$ and $pvc_{k}(G)$.
Note that $pc(G)=1$ if and only if $G$ is a complete graph. In
addition, by Proposition \ref{pro1} we have the following assertion:
if $diam(G)=1$, then $pc(G)=1$ and $pvc(G)=0$; if $diam(G)=2$, then
$pc(G)\geq 2$ and $pvc(G)=1$; if $diam(G)\geq3$, then $pc(G)\geq 2$
and $pvc(G)=2$. Thus, we have $pc(G)\geq pvc(G)$. For $k\geq 2$, the
following theorem shows that there exists an example graph $G$ such
that $pc_{k}(G)> pvc_{k}(G)$.

\begin{thm}\label{thm5}Let $G$ be a complete bipartite graph with
classes $U$ and $V$, where $U=\{u_{1},...,u_{t}\}$ and
$V=\{v_{1},...,v_{k}\}$ $(2\leq k<t)$. Then,  $pc_{k}(G)=t$ and $pvc_{k}(G)=2$.
\end{thm}

\begin{proof} Clearly, $pvc_{k}(G)=2$. Next we just need to prove $pc_{k}(G)=t$.
Since $G$ is a complete bipartite graph, it follows that
$\chi'(G)=\Delta$. Moreover, $pc_{k}(G)\leq \chi'(G)$ in $(1)$ and
$\Delta=t$. Then, $pc_{k}(G)\leq t$. If one colors the edges of $G$
with fewer than $t$ colors, then there exist two edges of
$\{v_1u_{i}:u_{i}\in U \}$, say $v_1u_1$ and $v_1u_2$, colored the
same. Thus, we can not have $k$ disjoint proper paths between
$u_{1}$ and $u_{2}$. Therefore, $pc_{k}(G)=t$.
\end{proof}

We observe that $pc_k(G)\geq pvc_k(G)$ for $k=1$. Moreover from
Theorem \ref{thm5}, we find an example such that $pc_k(G)> pvc_k(G)$
for $k\geq 2$. Note that $pc_{2}(G)=pvc_{2}(G)$ if $G$ is a cycle of
order $n\geq 4$. However, we cannot show whether there exists a
graph $G$ such that $pc_k(G)< pvc_k(G)$. Thus, we pose the following
problem.

\begin{prob} Let $k\geq2$. Does it hold that $pc_k(G)\geq pvc_k(G)$ for
any connected graph $G$?
\end{prob}

\section{Strong proper vertex-connection}

In \cite{LMS1}, Li et al. introduced the concept of strong rainbow
vertex-connection. A vertex-colored graph is {\it strong rainbow
vertex-connected}, if for any two vertices $u, v$ of the graph,
there exists a vertex-rainbow $u$-$v$ geodesic, i.e., a $u$-$v$ path
of length $d(u,v)$). For a connected graph $G$, the {\it strong
rainbow vertex-connection number} of $G$, denoted by $srvc(G)$, is
the smallest number of colors required to make $G$ strong rainbow
vertex-connected.

A natural idea is to introduce the concept of the strong proper
vertex-connection. A vertex-colored graph is {\it strong proper
vertex-connected}, if for any two vertices $u, v$ of the graph,
there exists a vertex-proper $u$-$v$ geodesic. For a connected graph
$G$, the {\it strong proper vertex-connection number} of $G$,
denoted by $spvc(G)$, is the smallest number of colors required to
make $G$ strong proper vertex-connected. Note that if $G$ is a
nontrivial connected graph, then $$0\leq pvc(G)\leq spvc(G)\leq
\min\{\chi(G), srvc(G)\}.\ \ \ \  \ \ \ \ \ \ \ \ \ \  \ \  (3)$$
The following results on $spvc(G)$ are immediate from definition.

\begin{pro}\label{pro2} Let $G$ be a nontrivial connected graph of order $n$. Then

$(a)$ \ $spvc(G)=0$ if and only if $G$ is a complete graph;

$(b)$ \ $spvc(G)=1$ if and only if $diam(G)=2$.

\end{pro}

It is easy to obtain the following consequences.

\begin{obser}\label{obser1}\

(1) \ $spvc(P_3)=1$ and $spvc(P_n)=2$ for $n\geq 4$;

(2) \ $spvc(C_4)=spvc(C_5)=1$, $spvc(C_n)=2$ for $n\geq 6$ even, and
$spvc(C_n)=3$ for $n\geq 7$ odd;

(3) \ $spvc(K_{s,t})=1$ for $s\geq2$ and $t\geq 1$;

(4) \ $spvc(K_{n_{1},n_{2},\ldots,n_{k}})=1$ for $k\geq 3$ and
$(n_{1},n_{2},\ldots,n_{k})\neq{(1,1,\ldots,1)}$;

(5) \ $spvc(W_{n})=1$ for $n\geq4$.
\end{obser}

In this section, sharp upper and lower bounds of $spvc(G)$ are given
for a connected graph $G$ of order $n$, that is, $0\leq spvc(G)\leq
n-2$. We also characterize the graphs of order $n$ such that
$spvc(G)=n-2, n-3$, respectively. Furthermore, we investigate the
relationship among the three vertex-coloring parameters, namely
$spvc(G), srvc(G)$ and $\chi(G)$ of a connected graph $G$.

\subsection{Bounds and characterization of extremal graphs}

The problem of finding bounds of $srvc(G)$ has been solved
completely by Li et al. \cite{LMS1}.

\begin{lem}\label{lem1}\cite{LMS1} Let $G$ be a connected graph of
order $n$ $(n\geq 3)$. Then $0\leq srvc(G)\leq n-2$. Moreover,
the bounds are sharp.
\end{lem}

\begin{lem}\label{lem2}\cite{LMS1} Let $G$ be a nontrivial
connected graph of order $n$. Then $srvc(G)=n-2$ if and only
if $G=P_n$.
\end{lem}

\begin{thm}\label{thm6} Let $G$ be a nontrivial connected graph
of order $n$. Then $0\leq spvc(G)\leq n-2$. Equality on the
right-hand side is attained if and only if $G\in\{P_3,P_4\}$.
\end{thm}

\pf By $(3)$ and Lemma \ref{lem1}, it is obvious that $0\leq
spvc(G)\leq srvc(G)\leq n-2$. On one hand, we know that
$spvc(P_3)=1=n-2$ and $spvc(P_4)=2=n-2$. On the other hand, if
$spvc(G)=n-2$, then $srvc(G)=n-2$. It follows that $G\in\{P_3,P_4\}$
from Observation \ref{obser1} and Lemma \ref{lem2}.\qed

\begin{thm}\label{thm7} Let $G$ be a nontrivial connected graph
of order $n$. Then $spvc(G)=n-3$ if and only if $G$ is one of the
twelve graphs in Figure 1.
\end{thm}

\begin{figure}[h,t,b,p]
\begin{center}
\scalebox{0.8}[0.8]{\includegraphics{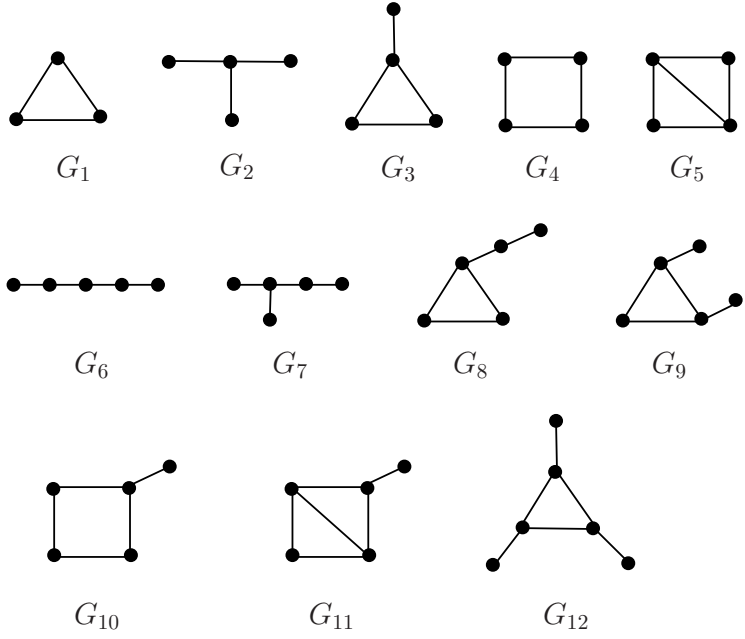}}
\end{center}
\caption{The twelve graphs in Theorem 7.}\label{Fig.1.}
\end{figure}

In order to prove Theorem \ref{thm7}, we need the lemma below.

\begin{lem}\label{lem3} If $G$ is a connected graph with order $n\geq 7$,
then $spvc(G)<n-3$.
\end{lem}

\pf Let $\Delta$ be the maximum degree of $G$.

Case 1. $\Delta=n-1$. Then, we have $diam(G)\leq 2$. By
Proposition \ref{pro2}, it follows that $spvc(G)\leq 1<n-3$.

Case 2. $\Delta=n-2$. Let $v$ be a vertex with the maximum degree
$\Delta$ and $N(v)=\{v_1,v_2,\ldots,v_{n-2}\}$ denote its
neighborhood. Let $v'$ be the only vertex not adjacent to $v$ and
$N(v')$ denote its neighborhood. Color the vertices of $N(v')$ with
color 1 and all other vertices with color 2. Now, we will show that
for any two vertices $u$ and $w$ in $G$, there exists a
vertex-proper geodesic between them. If $d(u,w)\leq 2$, then there
must be a vertex-proper $u$-$w$ geodesic. Since $G$ is connected,
there exists a vertex $v_i$ $(i\in\{1,2,\ldots,n-2\})$ such that
$v'$ is adjacent to $v_i$. If $d(u,w)=3$ for $u=v',w\in N(v)$ (or
$u\in N(v),w=v'$), then they are connected by a vertex-proper
geodesic $uv_ivw$ (or $uvv_iw$). For the other cases, $d(u,w)\leq
2$. Therefore, $spvc(G)\leq 2<n-3$.

Case 3. $\Delta=n-3$. Let $v$ be a vertex with the maximum degree
$\Delta$ and $N(v)=\{v_1,v_2,\ldots,v_{n-3}\}$ denote its
neighborhood. Let $v'$ and $v''$ be the vertices not adjacent to
$v$. Moreover, let $N(v')$ and $N(v'')$ denote the neighborhood of
$v'$ and $v''$, respectively.

Subcase 3.1. $N(v')\cap N(v'')\neq\emptyset$. Then, there exists one
vertex $v_i$ $(i\in\{1,2,\ldots,n-3\})$ such that $v'$ and $v''$ are
adjacent to $v_i$. Color the vertex $v_i$ with color 1 and all the
other vertices with color 2. Next, we shall show that there exists a
vertex-proper geodesic between any two vertices $u$ and $w$ in $G$.
If $d(u,w)=3$ for $u=v' \ ($or $v'')$ and $w\in N(v)$, then they are
connected by a vertex-proper path $uv_ivw$. If $d(u,w)=3$ for $u\in
N(v)$ and $w=v' \ ($or $v'')$, then there is a vertex-proper path
$uvv_iw$ between them. For the other cases, $d(u,w)\leq 2$.
Therefore, $spvc(G)\leq 2<n-3$.

Subcase 3.2. $N(v')\cap N(v'')=\emptyset$. Color the vertices of
$N(v')$ with color 1, the vertices of $N(v'')$ with color 2 and all
the others with color 3. Firstly, suppose that $v'v''$ is an edge of
$G$. If $N(v'')=\{v'\}$, then there exists a vertex $v_i$
$(i\in\{1,2,\ldots,n-3\})$ such that $v'$ is adjacent to $v_i$. By a
similar discussion of Case 2, we obtain that $G-v''$ is strong
proper vertex-connected. Furthermore, the color of $v'$ which is the
unique neighbor of $v''$, is distinct from others. Thus, there
exists a vertex-proper geodesic between any two vertices in $G$.
Similarly, the case that $|N(v'')|\geq 2$ can be proved. Now, assume
that $v'$ is not adjacent to $v''$. If there is an edge between
$N(v')$ and $N(v'')$, say $v_1v_2$ with $v_1\in N(v')$ and $v_2\in
N(v'')$, then $v'v_1v_2v''$ is a vertex-proper $v'$-$v''$ geodesic;
otherwise, $v'v_1vv_2v''$ is a vertex-proper $v'$-$v''$ geodesic. It
can be verified for any other pair of vertices in $G$ that there
exists a vertex-proper geodesic between them. Hence, $spvc(G)\leq
3<n-3$.

Case 4. $\Delta\leq n-4$. By $(3)$, we have $spvc(G)\leq\chi(G)$. If
$G$ is an odd cycle, then $\chi(G)=3$, and so $spvc(G)\leq 3<n-3$;
otherwise, $\chi(G)\leq \Delta$ by Brook' theorem \cite{B}, and so
$spvc(G)\leq\Delta<n-3$.

The proof is complete. \qed

Now, we are ready to prove Theorem \ref{thm7}.

{\it Proof of Theorem 7. } By Proposition \ref{pro2}, we obtain that
$spvc(G)=n-3$ for $G=G_i$ $(1\leq i\leq 5)$. If $G=G_i$ $(6\leq
i\leq 11)$, then $diam(G)\geq 3$, and so $spvc(G)\geq 2$. A
$2$-coloring of the vertices of $G=G_i$ $(6\leq i\leq 11)$ is shown
in Figure 2 to make $G$ strong proper vertex-connected. Thus,
$spvc(G)=n-3$ for $G=G_i$ $(6\leq i\leq 11)$. For the graph
$G_{12}$, color the three non-leaves with distinct colors. Then, we
can see that there exists a vertex-proper geodesic for any two
vertices. Hence, $spvc(G_{12})\leq 3$. However, if one colors the
vertices of $G_{12}$ with two colors, there exist two non-leaves
having the same color and then we can not find a vertex-proper
geodesic between the corresponding pendant vertices.

\begin{figure}[h,t,b,p]
\begin{center}
\scalebox{0.8}[0.8]{\includegraphics{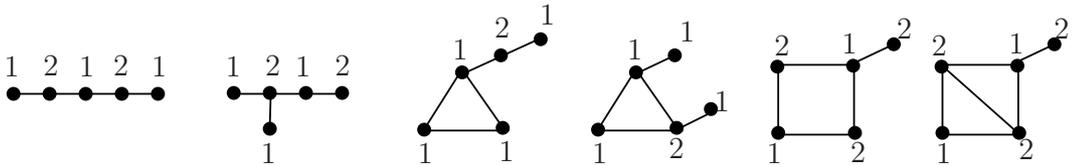}}
\end{center}
\caption{A $2$-coloring of vertices of $G=G_i$ $(6\leq i\leq 11)$.}\label{Fig.2.}
\end{figure}

It remains to verify the converse. Let $G$ be a connected graph of
order $n\geq 3$ such that $spvc(G)=n-3$. By Lemma \ref{lem3}, we
have that $n\in\{3,4,5,6\}$. Firstly, suppose $n=6$. Then,
$spvc(G)=3$. By Proposition \ref{pro2}, we get $diam(G)\geq 3$.
Moreover, $G$ contains a cycle; otherwise, $G$ is a tree and
$spvc(G)=2$. If $G\neq G_{12}$, it follows that $G$ contains a
subgraph isomorphic to one of the graphs $H_1,H_2,H_3,H_4,H_5,H_6$
in Figure 3, where a minimum vertex-coloring is also shown for each
graph to make it strong proper vertex-connected. Moreover, color the
remaining vertices of $G$ with any used color if there exist. It is
easy to check that $G$ is strong proper vertex-connected. Thus,
$spvc(G)\leq 2<3$, which is a contradiction. If $n=5$, then
$spvc(G)=2$. By Proposition \ref{pro2}, $diam(G)\geq 3$. However,
for $G\neq G_{i}$ $(6\leq i\leq 11)$, the diameter of $G$ is at most
two, which is a contradiction. Similarly, we deduce that $G=G_i$
$(2\leq i\leq 5)$ for $n=4$ and $G=G_1$ for $n=3$. \qed

\begin{figure}[h,t,b,p]
\begin{center}
\scalebox{0.8}[0.8]{\includegraphics{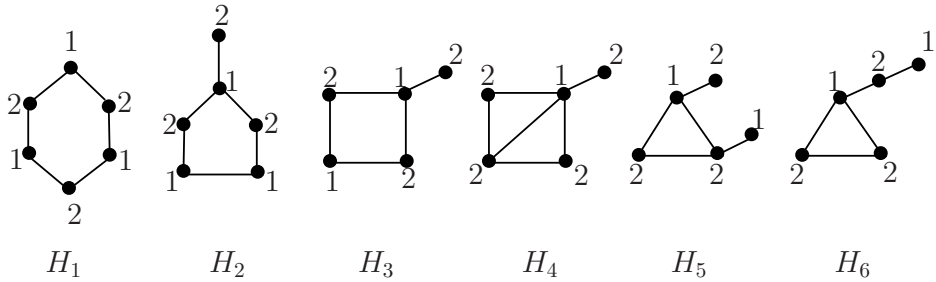}}
\end{center}
\caption{Subgraphs $H_1,H_2,\ldots,H_6$ in the proof of Theorem 7.}\label{Fig.3.}
\end{figure}

\subsection{Relationship of $spvc(G)$, $srvc(G)$ and $\chi(G)$}

By $(3)$, if $G$ is a nontrivial connected graph with $diam(G)\geq
3$ such that $spvc(G)=a$ and $srvc(G)=b$, then $2\leq a\leq b$.
Actually, this is the only restriction on the two parameters.

\begin{thm}\label{thm8} For every pair $a, b$ of integers
where $2\leq a\leq b$, there exists a connected graph $G$
such that $spvc(G)=a$ and $srvc(G)=b$.
\end{thm}

\pf Let $H$ be the corona $cor(K_a)$ of the complete graph
$K_a$ with $V(K_a)=\{v_1,v_2,\ldots,v_a\}$ and
$V(H\backslash{K_a})=\{v_1^{'},v_2^{'},\ldots,v_a^{'}\}$,
where $v_i^{'}$ is the corresponding pendant vertex of $v_i$
for $1\leq{i}\leq{a}$. Let $F=P_{b-a}$ with $V(F)=\{w_1,w_2,\ldots,w_{b-a}\}$.
Now let $G$ be the graph obtained from $H$ and $F$ by adding the edge $v_a^{'}w_1$.

Firstly, we show that $spvc(G)=a$. Define a vertex-coloring of $G$
as follows: Assign the color $j$ to $v_j$ for $1\leq{j}\leq{a}$ and
the color $1$ to $v_a^{'}$. For $1\leq{k}\leq{b-a}$, if $k$ is even,
assign the color $1$ to $w_k$; otherwise, assign the color $2$ to
$w_k$. We can see that every two vertices $x$ and $y$ are connected
by a vertex-proper $x$-$y$ geodesic. Hence, $spvc(G)\leq{a}$. If one
colors the vertices of $G$ with fewer than $a$ colors, then there
must be two vertices $v_s$ and $v_t$, where $1\leq s,t\leq a$, such
that they have the same color. However, we can not find a
vertex-proper geodesic between $v_s^{'}$ and $v_t^{'}$. Thus,
$spvc(G)=a$.

Next, we prove that $srvc(G)=b$. Define a vertex-coloring of $G$ by
assigning $(1)$ the color $j$ to $v_j$ for $1\leq{j}\leq{a}$, $(2)$
the color $a+1$ to $v_a^{'}$ and $(3)$ the color $a+1+k$ to $w_k$
for $1\leq{k}\leq{b-a-1}$. Since all non-leaves of $G$ have distinct
colors, $G$ is strong rainbow vertex-connected. Hence,
$srvc(G)\leq{b}$. If one colors the vertices of $G$ with fewer than
$b$ colors, then there must be two vertices of
$\{v_1,v_2,\ldots,v_a,v_a^{'},w_1,w_2,\ldots,w_{b-a-1}\}$ having the
same color. Furthermore, the colors of
$\{v_a^{'},w_1,w_2,\ldots,w_{b-a-1}\}$ must be distinct since there
is only one path between $v_a$ and $w_{b-a}$. If the colors of $v_i$
and $v_j$ $(1\leq{i,j}\leq{a})$ are the same, then there does not
exist a vertex-rainbow geodesic between $v_i^{'}$ and $v_j^{'}$. If
the colors of $v_i$ and $w_k$, where $1\leq{i}\leq{a-1}$ and
$1\leq{k}\leq{b-a-1}$, are the same, then we can not find a
vertex-rainbow geodesic between $v_i^{'}$ and $w_{k+1}$. If the
colors of $v_a$ and $v_k$ $(1\leq{k}\leq{b-a-1})$ are the same, then
there does not exist a vertex-rainbow geodesic between $v_1$ and
$w_{k+1}$. Thus $srvc(G)=b$. \qed

We saw in $(3)$ that if $G$ is a nontrivial connected graph with
$diam(G)\geq 3$ which is not an odd cycle such that
$spvc(G)=a,\chi(G)=b$ and $\Delta(G)=c$, then $2\leq a\leq b\leq c$.
In fact, this is the only restriction on the three parameters.

\begin{thm}\label{thm9} For every triple $a, b, c$ of integers
where $2\leq a\leq b\leq c$, there exists a connected graph $G$ such
that $spvc(G)=a, \chi(G)=b$ and $\Delta(G)=c$.
\end{thm}

\pf Let $H=K_b$ with $V(K_b)=\{v_1,v_2,\ldots,v_b\}$. Then, add
$c-b+1$ pendant vertices, denoted by
$\{v_1^{1},v_1^{2},\ldots,v_1^{c-b+1}\}$, to $v_1$, and a pendant
vertex $v_i^{1}$ to $v_i$ for $2\leq{i}\leq{a}$. Write $G$ as the
resulting graph. It is easy to see that the maximum degree of $G$ is
$c$, i.e., $\Delta(G)=c$.

In the following, we first show that $spvc(G)=a$. Define a
vertex-coloring of $G$ by assigning the color $j$ to $v_j$ for
$1\leq{j}\leq{a}$. It is easy to check that every two vertices $x$
and $y$ are connected by a vertex-proper $x$-$y$ geodesic. Hence,
$spvc(G)\leq{a}$. If one colors the vertices of $G$ with fewer than
$a$ colors, then there must be two vertices $v_j$ and $v_k$
$(1\leq{j,k}\leq{a})$ such that they have the same color. However,
we can not find a vertex-proper geodesic between $v_j^{1}$ and
$v_k^{1}$. Thus $spvc(G)=a$.

Next, we show that $\chi(G)=b$. Define a vertex-coloring of $G$ by
assigning $(1)$ the color $j$ to $v_j$ ($1\leq{j}\leq{b}$), $(2)$
the color $j-1$ to $v_j^{1}$ ($2\leq{j}\leq{a}$) and $(3)$ the color
$2$ to $v_1^{k}$ ($1\leq{k}\leq{c-b+1}$). We can see that any two
adjacent vertices have distinct colors. Hence, $\chi(G)\leq{b}$. If
one colors the vertices of $G$ with fewer than $b$ colors, then
there must exist two adjacent vertices $v_j$ and $v_k$
$(1\leq{j,k}\leq{b})$ such that they have the same color. Thus
$\chi(G)=b$. \qed

\end{document}